\documentclass[10pt, reqno]{amsart}
\usepackage{amsmath,amssymb,latexsym, amsfonts, amscd, amsthm}

\def\A{\mathbb{A}}                          \def\F{\mathbb{F}}                                        \def\N{\mathbb{N}}       \def\PP{\mathbb{P}}     \def\Q{\mathbb{Q}}                
                        
\def\Z{\mathbb{Z}}

        \def\cC{\mathcal{C}}

\def\GK{K_0(\Var_k)}

\DeclareMathOperator{\Spec}{Spec}

\DeclareMathOperator{\Var}{Var}

\DeclareMathOperator{\Sym}{Sym}
\DeclareMathOperator{\Fr}{Fr}

\DeclareMathOperator{\gh}{gh}

\newtheoremstyle{thm}{9pt}{9pt}{\itshape}{}{\itshape \bfseries}{.\,---}{.5em}{}
\theoremstyle{thm}
\theoremstyle{plain}
\newtheorem{theorem}{Theorem}

\newtheorem{lemma}[theorem]{Lemma}

\newtheorem{corollary}[theorem]{Corollary}

\theoremstyle{definition}
\newtheorem{definition}[theorem]{Definition}
\newtheorem{proposition}[theorem]{Proposition}
\newtheorem{example}[theorem]{Example}

\newtheorem*{Claim*}{Claim}
\newtheorem*{remark}{Remark}

\def\skipline{
\vspace{4mm}%
}
\numberwithin{theorem}{section}
\usepackage[toc,page]{appendix}
\usepackage{tikz}
\usepackage{bm}
\usepackage{hyperref}%
\usepackage[all]{xy}       
    \SelectTips{cm}{10}     
    \everyxy={<2.5em,0em>:} 

\begin{document}
\title[A MacDonald Formula for Varieties Over Finite Fields]{A MacDonald Formula for Zeta Functions of Varieties Over Finite Fields}
\author{Jonathan A. Huang}
\address{Department of Mathematics and Computer  Science, Wesleyan University, Middletown, CT 06459-0128}
\email{jhuang03@wesleyan.edu}
\date{January 11, 2018}

\begin{abstract}
We provide a formula for the generating series of the zeta function $Z(X,t)$ of symmetric powers $\Sym^n X$ of varieties over finite fields.  This realizes $Z(X,t)$ as an exponentiable motivic measure whose associated Kapranov motivic zeta function takes values in $W(R)$ the big Witt ring of $R=W(\Z)$.   We apply our formula to compute $Z(\Sym^nX,t)$ in a number of explicit cases.  Moreover, we show that all $\lambda$-ring motivic measures have zeta functions which are exponentiable.  In this setting, the formula for $Z(X,t)$ takes the form of a MacDonald formula for the zeta function.
\end{abstract}

\maketitle

%

\section{Introduction}
A remarkable formula of MacDonald \cite{MacDonald} provides a closed expression for the generating series of the Poincar\'{e} polynomial of the symmetric powers $\Sym^n X$ of a space $X$. Let $X$ be a compact complex manifold of dimension $m$, and recall the Poincar\'{e} polynomial is defined as $P(X,z) = \sum_i (-1)^i b_i(X) z^i$, where $b_i=b_i(X) = \dim_\Q H^i(X,\Q)$ are the Betti numbers. The MacDonald formula for $P(X,z)$ is
\begin{equation} \label{macd}
\sum_{n \geqslant 0} P(\Sym^n X, z) t^n =
\frac{(1-z^1t)^{b_1}(1-z^3t)^{b_3}\cdots (1-z^{2m-1}t)^{b_{2m-1}}}
{(1-t)^{b_0}(1-z^2t)^{b_2}\cdots(1-z^{2m}t)^{b_{2m}}}.
\end{equation}
Thus the Poincar\'{e} polynomial $P(\Sym^n X,z)$ may be expressed directly in terms of invariants associated to $X$.  A similar formula for the Euler characteristic is recovered by setting $z=1$ as  $P(X,1) = \chi(X)$.

In this short note we prove an analogous formula for the Hasse-Weil zeta function $Z(X,t)$ of varieties over finite fields.  To provide such a formula, we consider the zeta function $Z(X,t)$ as an element in the Witt ring $W(\Z)$.  As noted by Ramachandran \cite{Ramachandran15}, $Z(X,t)$ takes form of an Euler-Poincar\'{e} characteristic as an alternating sum of \emph{Teichm\"{u}ller elements} $[\alpha]$:
$$
Z(X,t) = \sum_{i,j} (-1)^{i+1} [\alpha_{ij}] \in W(\overline{\Q_\ell})
$$
where $\alpha_{ij}$ is the $j$th (inverse) eigenvalue of the Frobenius operator acting on  the $i$th \'{e}tale cohomology.  By definition, the zeta function $Z(X,t)$ takes values in the ring $W(\Z) \subset W(\overline{\Q_\ell})$.

\skipline

Our main result is a MacDonald formula for $Z(X,t)$:
\begin{theorem}\label{mainthm}
Given a quasi-projective variety $X \in \Var_{\F_q}$, we have
$$
\sum_{n \geqslant 0} Z(\Sym^n X, t) u^n =  \sum_{i,j} (-1)^{i+1}[[\; [\alpha_{ij}] \;]] \in W(W(\overline{\Q_\ell}))
$$
where the $[[\; [\alpha_{ij}] \;]]$ are double Teichm\"{u}ller elements.
\end{theorem}

\skipline

We interpret this result in the setting of motivic zeta functions, specifically in the context of $\lambda$-ring-valued motivic measures.  Let $\GK$ be the Grothendieck ring of varieties and consider an $A$-valued motivic measure $\mu:\GK \rightarrow A$.  In \cite{Kapranov}, Kapranov associates to $\mu$ a motivic zeta function
$$\zeta_\mu(X,t)  = \sum_{n\geqslant 0} \mu(\Sym^n X) t^n,$$
an invertible power series with coefficients in $A$.  Following Ramachandran \cite{Ramachandran15}, we say that $\mu$ \emph{exponentiates} or \emph{can be exponentiated} if $\zeta_\mu$ takes values in the big Witt ring $W(A)$.  That is, if the zeta function $\zeta_\mu$ reflects the product structure in $\GK$ via the Witt product in $W(A)$. A particularly interesting case is that of $A$-valued motivic measures where $A$ is a $\lambda$-ring.  It is shown by Ramachandran and Tabuada \cite{RamaTabuada} that if such a measure is a pre-$\lambda$-ring map on $\GK$, then it exponentiates (see Proposition \ref{NRTaba}).  We refer to such measures $\mu$ as $\lambda$-measures.

The zeta function of such a measure is a ring homomorphism into $W(A)$ and thus induces a zeta function measure $\mu_Z([X]) = \zeta_\mu(X,t)$.  We prove that if $\mu$ is a $\lambda$-measure, then $\mu_Z$ is also a $\lambda$-measure (Theorem \ref{zetaofzeta}); this involves a study of the $\lambda$-ring structure of the big Witt ring $W(A)$.  The classical MacDonald formulae show that the Euler characteristic and Poincar\'{e} polynomial define $\lambda$-measures.  The formula in  Theorem \ref{mainthm} shows that the zeta function $Z(X,t)$ for varieties over finite fields defines a $\lambda$-measure.

Analogous MacDonald formulae can be proved in other contexts.  Let $\cC$ be a $k$-linear tensor category, in the sense of a $k$-linear additive pseudo-abelian symmetric monoidal category where $\otimes$ is $k$-linear.  As shown by Heinloth \cite{Heinloth07}, the Grothendieck ring $K_0(\cC)$ is a $\lambda$-ring.  Del Ba\~{n}o Rollin \cite{Rollin01} and Maxim-Sch\"{u}rmann \cite{MaximSchurmann} prove examples of closed formulae for various generating series of measures taking values in various $K(\cC)$; these follow directly from the opposite $\lambda$-ring structure map $\sigma_t$.  The motivic measures described in \cite{Rollin01} and \cite{MaximSchurmann} are also $\lambda$-measures.

After reviewing some preliminaries on $\lambda$-rings and Witt rings in Section \ref{lambdarings}, we prove the formula for zeta functions in Section \ref{FiniteFieldCase} and provide some explicit computations of zeta functions of various $\Sym^n X$ in Section \ref{examples}.  Finally in Section \ref{lambdaringmeasures}, we provide some context for these results in the theory of $\lambda$-ring-valued motivic measures and explore related results for motivic zeta function measures. 

\subsection*{Acknowledgements}
The author would like to thank Niranjan Ramachandran for suggesting this project and for eagerly and openly sharing his ideas and insights. The author is also indebted to the Department of Mathematics at the University of Maryland, College Park---in particular to helpful discussions with and encouragement from Larry Washington and Jonathan Rosenberg.

\section{Lambda rings and Witt rings}\label{lambdarings}
We begin by describing the $\lambda$-ring structure on the big ring of Witt vectors. Our presentation here of the big Witt ring follows Bloch \cite{Bloch} and Ramachandran \cite{Ramachandran15}.  For a more on $\lambda$-rings, the reader may consult Knutson \cite{Knutson}, Yao \cite{Yau}, and Hazewinkel \cite{Hazewinkel}.

Let $A$ be a commutative ring with identity.  Denote by $\Lambda(A)$ the group of invertible power series $(1+tA[[t]],\times)$ under usual multiplication of power series.  For a ring homomorphism $f:A \rightarrow A'$, let $\Lambda_f:\Lambda(A) \rightarrow \Lambda(A')$ be the obvious induced map.  Recall the ghost map $\gh:\Lambda(A) \rightarrow A^{\N}$ is defined as, for $P \in \Lambda(A)$,
$$
\gh(P) = (b_1,b_2,\ldots) \quad \textrm{where} \quad \frac{t}{P}\frac{dP}{dt} = \sum_{n=1}^\infty b_n t^n.
$$
This represents the power series $P(t)$ by a series of coordinates $(b_1,b_2,\ldots)$ called ghost coordinates $\gh_n(P) = b_n$.  The ghost map is a functorial group homomorphism where $ A^{\N}$ has a pointwise addition.  Moreover, the relation between the power series coefficients and ghost coordinates can be made explicit:
\begin{lemma}\label{newton}
For $P(t) = \sum_n a_n t^n \in \Lambda(A)$ and $b_n = \gh_n(P(t))$, we have
$$
na_n = b_n + a_1b_{n-1} + \cdots a_{n-1}b_{1}.
$$
This relation uniquely determines $P(t)$ in terms of its ghost coordinates $b_n$ in the case where $A$ has no $\Z$-torsion.
\end{lemma}
\begin{proof}
By definition, $\dfrac{t}{P} \dfrac{dP}{dt} = \sum_{n=1}^\infty b_nt^n$ so that
\begin{eqnarray*}
P(t)\left( \sum_{n=1}^\infty b_nt^n \right) &=& t \frac{dP}{dt} \\
(1+ a_1t + a_2t^2 + \cdots )(b_1t + b_2t^2 + \cdots) &=& t(a_1 + 2a_2t + 3a_3t^2 \cdots) \\
&=& a_1t + 2a_2t^2 + 3a_3t^3 + \cdots
\end{eqnarray*}
and we simply identify coefficients.
\end{proof}

The big Witt ring $W(A)$ (also called the universal ring of Witt vectors) is the ring whose underlying additive group is $\Lambda(A)$ and whose multiplication $*_W$ is such that
$$
(1-at)^{-1} *_W (1-bt)^{-1} = (1-abt)^{-1} \quad \textrm{for } a,b \in A
$$
and the association $A \mapsto W(A)$ is functorial.  This suffices to define a commutative ring structure on $W(A)$.  The elements $[a] = (1-at)^{-1}$ above are called Teichm\"{u}ller elements.

The ghost map $\gh:W(A) \rightarrow A^{\N}$ is a ring homomorphism where $A^{\N}$ has a pointwise product.  That is, for $P,Q \in W(Z)$, $gh_n(P*_W Q) = \gh_n(P)\gh_n(Q)$ and thus in the case where $A$ has no $\Z$-torsion, multiplication is determined by pointwise product of ghost coordinates.   Note that for Teichm\"{u}ller elements, $\gh_n([a]) = a^n$.

For each positive integer $n$, the Frobenius map $\Fr_n:W(A) \rightarrow W(A)$  is the ring endomorphism defined as
$$
\Fr_n(P(t)) = \prod_{\zeta^n = 1} P(\zeta t^{1/n}).
$$
On Teichm\"{u}ller elements $\Fr_n([a]) = [a^n]$ and on ghost coordinates, $\gh_m (Fr_n(P))= \gh_{mn}( P)$.  Clearly, we have $\Fr_n \circ \Fr_m = \Fr_{nm}$.

Recall that a pre-$\lambda$-ring $(A,\lambda_t)$ is a commutative ring with identity $A$ equipped with a group homomorphism  called the structure map $\lambda_t:A \rightarrow \Lambda(A)$, where the coefficients are denoted as $\lambda_t(a) = \sum_{n\geqslant 0} \lambda^n(a) t^n$.  Note in particular that $\lambda^0(a) = 1$ and $\lambda^n(a+b) = \sum_{i+j=n} \lambda^i(a)\lambda^j(b)$.  Here, $\Lambda(A)$ is a pre-$\lambda$-ring with multiplication and a canonical structure map defined by certain universal polynomials (see Knutson \cite[p. 13]{Knutson}).  A map of pre-$\lambda$-rings $f:A \rightarrow A'$ is a ring homomorphism respecting the $\lambda$-ring structure maps.  If $\lambda_t:A \rightarrow \Lambda(A)$ is  a ring homomorphism, then $A$ is called a $\lambda$-ring; a $\lambda$-ring map is a map of pre-$\lambda$-rings.  $\Lambda(A)$ is a $\lambda$-ring.

The Witt ring $W(A)$ also has a canonical $\lambda$-ring structure\footnote{Some authors refer to a $\lambda$-ring as a special $\lambda$-ring and a pre-$\lambda$-ring as a $\lambda$-ring.} denoted $\bm{\lambda_u}$ (see Knutson \cite[p. 18]{Knutson}).  On Teichm\"{u}ller elements, it is given by
$$
\bm{\lambda_u}([a]) = [1] + [a] u \in \Lambda(W(A))
$$
where $\Lambda(W(A)) = 1 + u W(A)[[u]]$.  This $\lambda$-ring structure behaves well  in the case where $A$ itself is a $\lambda$-ring.  In particular, we have
\begin{proposition}\label{sigma}
If $(A, \lambda_t)$ is a $\lambda$-ring, then the opposite $\lambda$-ring map $\sigma_t = \lambda_{-t}^{-1}$, with $\sigma_t(a) := \sum_{n\geqslant 0} \sigma^n(a) t^n$  is a $\lambda$-ring homomorphism $\sigma_t:A \rightarrow W(A)$.  In particular,
$$
\Lambda_{\sigma_t}(\sigma_t(a)) = \bm{\sigma_u}(\sigma_t(a))
$$
\end{proposition}
\begin{proof}
Since $A$ is a $\lambda$-ring, $\lambda_t:A \rightarrow \Lambda(A)$ is a $\lambda$-ring map.  The Artin-Hasse exponential $\iota:\Lambda(A) \rightarrow W(A), P(t)\mapsto P(-t)^{-1}$ is a $\lambda$-ring map, and $\iota \circ \lambda_t = \sigma_t$.
\end{proof}

The following examples will be useful:
\begin{itemize}
\item The ring $A = \Z$ is a $\lambda$-ring with $\lambda_t(a) = (1+t)^a$.  Here the map $\sigma_t$ is given by
$$
\sigma_t(a) = (1-t)^{-a}, \quad \quad \sigma^n(a) = \left(\begin{array}{c} n+a+1 \\ a \end{array}  \right).
$$
\item For a $\lambda$-ring $A$, the polynomial ring $A[z]$ can be given a natural $\lambda$-ring structure by setting $\lambda_t(z) = (1+zt)$.  Here the map $\sigma_t$ is given by
     $$
     \sigma_t(z) = (1-zt)^{-1} =[z], \quad \quad \sigma^n(z)=z^n.
     $$
      Note that for $A = \Z$ this means that $\sigma_t(az) = (1-zt)^{-a} = a[z]$.
\item For arbitrary $A$, consider the canonical $\lambda$-ring structure on $W(A)$.  The opposite $\lambda$-ring structure on $W(A)$ is a $\lambda$-ring map $\bm{\sigma_u}:W(A) \rightarrow W(W(A))$.  On Teichm\"{u}ller elements $[a] \in W(A)$ we have $\bm{\sigma_u}([a]) = ([1]-[a]u)^{-1}$.  We denote this by $[[\;[a]\;]]$, the \emph{double Teichm\"{u}ller element.}
\end{itemize}


\section{Varieties over Finite Fields}\label{FiniteFieldCase}
Let $X$ be a variety over a finite field $\F_q$  (reduced scheme of finite type over $\Spec \F_q$). The Hasse-Weil zeta function (or simply the zeta function) of $X$ is
\begin{equation} \label{zetafunction}
Z(X,t) = \prod_{x \in |X|} (1-t^{\deg x})^{-1}
\end{equation}
a power series over $\Z$, where the product ranges over closed points $x \in |X|$.  This may also be expressed as
\begin{equation} \label{weilzetafunction}
Z(X,t) = \exp \left[\sum_{r=1}^\infty N_r(X) \frac{t^r}{r} \right]
\end{equation}
where $N_r(X) = \# X(\F_{q^r})$ the number of points over $\F_{q^r}$.

Due to the work of Grothendieck and others on the Weil conjectures, the zeta function can be written in terms the action of the Frobenius on (compactly supported) \'{e}tale cohomology:
\begin{equation}\label{zetapoly}
Z(X,t) = \frac{\prod_i P_{2i+1}(t)}{\prod_i P_{2i}(t)}, \quad P_i(t) = \prod_j (1-\alpha_{ij} t)^{-1}
\end{equation}
where $\alpha_{ij}$ are the  inverse eigenvalues of the Frobenius action $\Phi$ on $H_{et,c}^i(\overline{X};\Q_\ell)$.  These are so-called Weil $q$-numbers $|\alpha_{i,j}|= q^{r/2}$ for some $r \leqslant i$ (Deligne) and  $Z(X,t)$ is a rational function (Dwork).

We wish to consider the zeta function $Z(X,t)$ as an element in the big Witt ring $W(\Z)$.  The following is shown in \cite[Thm 2.1]{Ramachandran15}:
\begin{proposition}\label{weilzetawitt}
For $X$ and $Y \in \Var_{\F_q}$, the zeta function $X \times_{\F_q} Y$ is given by the Witt product
$$
Z(X\times_{\F_q} Y, t) = Z(X,t) *_W Z(Y,t).
$$
Moreover, the Frobenius operator $\Fr_r$ on $W(\Z)$ corresponds to base change $X_{\F_{q^r}}$, so that $\Fr_r(Z(X/\F_q,t)) = Z(X/\F_{q^r},t).$
\end{proposition}
\begin{proof}[Sketch]
Notice that from (\ref{weilzetafunction}), $\gh_n(Z(X,t)) = N_n(X)$.  As this is multiplicative, the zeta function takes products in the Witt ring.  Finally, recall that $\gh_n \circ \Fr_r = \gh_{nr}$ and note that $\gh_n(X_{\F_{q^r}}) = N_{nr}(X) = \gh_{nr}(X/\F_q)$.
\end{proof}
The zeta function can be written in the Witt ring $W(\Z)$ as
$$
Z(X,t) = \sum_{i,j} (-1)^{i+1}[\alpha_{ij}]
$$
where the sum is taking place in the Witt ring and $[\alpha]$ is the Teichm\"{u}ller element $(1-\alpha t)^{-1}$.  This directly follows from the presentation in (\ref{zetapoly}) as an alternating product of polynomials $P_i = (1-\alpha_{ij}t)^{-1}$ and Proposition \ref{weilzetawitt} above.  Although each Teichm\"{u}ller element $[\alpha_{ij}]$ is in $W(\overline{\Q}_\ell)$, the sum nevertheless lies in the subring $W(\Z)$.\\

We now restate our main result.
\begin{theorem}
Given a quasi-projective variety $X \in \Var_{\F_q}$, we have
\begin{equation}\label{maineq}
\sum_{n \geqslant 0} Z(\Sym^n X, t) u^n =  \sum_{i,j} (-1)^{i+1}[[\; [\alpha_{ij}] \;]] \in W(W(\overline{\Q_\ell}))
\end{equation}
\end{theorem}
This can be viewed as a closed product formula for the generating series for zeta functions of symmetric powers.

\subsection{Proof of the Main Theorem}
Recall the result on Newton's identities from Lemma \ref{newton}: in $\Lambda(A)$, for $P(t) = \sum_n a_n t^n$ and $b_n = \gh_n(P(t))$, we have
$$
na_n = b_n + a_1b_{n-1} + \cdots a_{n-1}b_{1}
$$
Recursively, this gives a way to recover $a_n$ from the ghost coordinates $b_1,b_2,\ldots,b_n$.  That is, $a_n$ can be written purely in terms of $b_i$ for $i = 1, 2, \ldots n$ by replacing each $a_i$ in the above relation.  Let us call this relation $\phi$ so that
$$
a_n = \phi(b_1,b_2,\ldots,b_n).
$$

\begin{lemma}\label{frobnewton}
For $X$ a quasi-projective variety over $\F_q$,
\begin{equation}\label{frobnewtoneq}
N_r(\Sym^n X) = \phi(N_r(X),N_{2r}(X),\ldots,N_{nr}(X))
\end{equation}
\end{lemma}
\begin{proof}
Over $\overline{\F_q}$, the symmetric power $\Sym^n X$ parametrizes effective zero cycles.  These are exactly the degree $n$ effective cycles on $X$.  Thus the zeta function $Z(X,t)$ may be written as
\begin{equation}\label{zetaforlemma}
Z(X,t) =\sum_n N_1(\Sym^n X)t^n.
\end{equation}
As $\gh_r(Z(X,t)) = N_r(X)$, this implies that
$$
N_1(\Sym^n X) = \phi(N_1(X),N_2(X),\ldots,N_n(X)).
$$
Apply the Frobenius operator $\Fr_r$ for $W(\Z)$ to equation (\ref{zetaforlemma}).  Equation (\ref{frobnewtoneq}) then follows from the observation that on ghost coordinates, $\gh_n (\Fr_r(P)) = \gh_{nr}(P)$ for $P \in W(\Z)$ and from Proposition \ref{weilzetawitt}, $\Fr_r Z(X/\F_q,t) = Z(X/\F_{q^r},t)$.
\end{proof}

\begin{proof}[Proof of Theorem \ref{mainthm}]
Denote the ghost map on $W(W(\Z))$ by $\gh^u$ and let $\beta_n$ be the $n$th ghost coordinate of the sum double Teichm\"{u}ller elements in (\ref{maineq}):
$$
\beta_n =\gh^u_n\left(\sum_{i,j} (-1)^{i+1}[[\, [\alpha_{ij}] \,]]\right) = \sum_{i,j} (-1)^{i+1} [\alpha_{ij}]^n \in W(\Z).
$$
We wish to show
$$
Z(\Sym^nX,t) = \phi( \beta_1,\beta_2,\ldots, \beta_n).
$$
This relation is taking place in $W(\Z)$.  We now use the ghost map $\gh^t$ on $W(\Z)$.  It suffices to show for all $r$,
$$
\gh_r^t Z(\Sym^n X,t) = \phi( \gh_r^t \beta_1, \gh_r^t \beta_2,\ldots, \gh_r^t \beta_n).
$$
We have
$$
\gh^t_r\beta_n = \gh^t_r  \left( \sum_{i,j} (-1)^{i+1} [\alpha_{ij}]^n  \right) = \sum_{i,j} (-1)^{i+1} \alpha_{ij}^{nr} = N_{nr}(X) \in \Z.
$$
Moreover, $\gh_r^t Z(\Sym^n X,t) = N_r(\Sym^n X)$.  Now recall from Lemma \ref{frobnewton}
$$
N_r(\Sym^n X) = \phi(N_r(X),N_{2r}(X),\ldots,N_{nr}(X)).
$$
As this holds for all $r$, the proof follows.
\end{proof}

\section{Examples}\label{examples}
We apply the generating series formula to compute $Z(\Sym^n X,t)$ for various cases of varieties $X$ over finite fields.  These serve to verify the formula and demonstrate the facility of writing the zeta function in the Witt ring.  Note that below, sums of Teichm\"{u}ller elements are taking place in the appropriate Witt ring.

\subsection{Affine and projective space}
Consider $n$-dimensional affine space $\A^n$ and $n$-dimensional projective space $\PP^n$ over $\F_q$.  Then $
N_r(\A^n) = q^{nr}$ and we have
$$
Z(\A^n,t) = \frac{1}{(1-q^n t)}
$$
It is also easy to show that $N_r(\PP^n) = q^{nr}+ q^{(n-1)r} + \cdots +q^r + 1$ so that
$$
 Z(\PP^n,t) = \frac{1}{(1- t)(1-qt)\cdots(1-q^mt)}
$$
In the Witt ring $W(\Z)$, these zeta functions are $Z(\A^n,t) = [q^n]$ and  $Z(\PP^n,t) =  [q^m] + [q^{m-1}] + \cdots + [q] + [1]$.

Recall that $\Sym^n \A^1 = \A^n$ and $\Sym^n \PP^1 = \PP^n$, and that
$$
Z(\A^1,t) = \frac{1}{1-qt} = [q] \quad \textrm{and} \quad Z(\PP^1,t) = \frac{1}{(1-t)(1-qt)} = [1] + [q]
$$
The formula provided in Theorem \ref{mainthm} predicts that $Z(\Sym^n \A^1,t)$ is the coefficient of $u^n$ in $[[\;[q]\;]] \in W(W(\Z))$.  Similarly, that  $Z(\Sym^n \PP^1,t)$ is the coefficient of $u^n$ in $[[\;[1]\;]] + [[\;[q]\;]] \in W(W(\Z))$. We have
$$
[[\;[q]\;]] = \frac{1}{[1]-[q]u} = [1] + [q]u + [q^2] u^2 + \cdots
$$
and
$$
[[\;[1]\;]] + [[\;[q]\;]]= \left(\frac{1}{[1]-[1]u}\right) \left(\frac{1}{[1]-[q]u}\right) = [1] + ([1]+[q])u + \cdots
$$
which agrees with the zeta functions $Z(\A^n,t)$ and $Z(\PP^n,t)$ described above:

\subsection{Elliptic curves}
Let $E$ be an elliptic curve over $\F_q$. In this case, we have
$$
Z(E,t) = \frac{(1-\alpha t)(1 - \beta t)}{(1-t)(1-qt)}
$$
where $\alpha + \beta = a \in \Z$ and $\alpha\beta = q$.  Written in the Witt ring, this appears as
$$
Z(E,t) = [1] - [\alpha] - [\beta] + [q],
$$
where $[\alpha] = (1-\alpha t)^{-1}$ the Teichm\"{u}ller element.  The formula provided by Theorem \ref{mainthm} is
$$
Z(\Sym^n E,t) = \textrm{coefficient of $u^n$ in} \left[\frac{(1-[\alpha]u)(1-[\beta] u))}{(1-[1] u)(1-[q]u))} \right].
$$
The results in the section in fact work for any curve $C$ over $\F_q$.  In this general case, the eigenvalues $\alpha$ and $\beta$ for the action on $H^1(\overline{C},\Q_\ell)$ come in pairs $\alpha_i$ and $\beta_i$ for $i = 1,2,\ldots,g$ where $g$ is the genus.  In fact, the case of symmetric powers of smooth projective curves was worked out by MacDonald in \cite{MacDonald62} in 1962.  Here we work in the elliptic curve case for simplicity.

\subsection{Symmetric powers of affine and projective space}
We use the formula in Theorem \ref{mainthm} to compute the zeta functions $Z(\Sym^n \A^m,t)$ and $Z(\Sym^n \PP^m,t)$.  Note that these varieties are no longer smooth once $m>1$. Recall that
$$
Z(\A^m,t) = \frac{1}{(1-q^mt)} = [q^m] \quad \textrm{and} \quad Z(\PP^m,t) = \frac{1}{(1-t)(1-[q^m])} = [1] + [q^m]
$$
Then by the formula we have for affine space
$$
Z(\Sym^n \A^m, t) = \textrm{ coefficient of $u^n$ in } [[\;[q^m]\;]] = [q^{nm}],
$$
which agrees with the fact that $[\Sym^n \A^m]= [\A^{nm}]$.  For projective space,
$$
Z(\Sym^n \PP^m, t) = \textrm{ coefficient of $u^n$ in } [[\;[1]\;]]+[[\;[q]\;]]+\cdots + [[\;[q^m]\;]].
$$
Note that this coefficient is not $[1] + [q]+\cdots + [q^{nm}]$ as $\Sym^n \PP^m$ is not $\PP^{nm}$.  For instance, consider $n=2$ and $m=2$,
$$
Z(\Sym^2 \PP^2, t) = [1] + [q] + 2[q^2] + [q^3] + [q^4]
$$
whereas
$$
Z(\PP^4,t) = [1] + [q]+ [q^2]+[q^3] + [q^4].
$$

\subsection{Product of elliptic curves}
Consider $X = E_1 \times E_2$ a product of elliptic curves over $\F_q$.  We use the formula in Theorem \ref{mainthm} to compute $Z(\Sym^n X,t)$ in terms of $Z(E_1,t)$ and $Z(E_2,t)$.  Let
$$
Z(E_1,t) = \frac{(1-\alpha_1 t)(1 - \beta_1 t)}{(1-t)(1-qt)} \quad \textrm{and} \quad
Z(E_2,t) = \frac{(1-\alpha_2 t)(1 - \beta_2 t)}{(1-t)(1-qt)}
$$
and recall that in the Witt ring, these can be written as
$$
Z(E_1,t) = [1] - [\alpha_1] - [\beta_1] + [q] \quad\textrm{and} \quad
Z(E_2,t) = [1] - [\alpha_2] - [\beta_2] + [q].
$$
Then since the zeta function takes products in the Witt ring, we may compute $Z(E_1 \times E_2,t)$ as
\begin{eqnarray*}
Z(E_1\times E_2,t) &=& ([1] - [\alpha_1] - [\beta_1] + [q] )* ([1] - [\alpha_2] - [\beta_2] + [q]) \\
&=& [1]- [\alpha_1] - [\alpha_2] - [\beta_1] - [\beta_2] + [\alpha_1\beta_2]+[\alpha_2\beta_1] +[\alpha_1\alpha_2] \\
&& + [\beta_1\beta_2] - [\alpha_1q] - [\alpha_2q] - [\beta_1q] - [\beta_2q] + [q^2].
\end{eqnarray*}
Thus the motivic zeta function generating series is
\begin{eqnarray*}
\zeta_Z(E_1\times E_2,u) &=& \sum_{n=0}^\infty Z(\Sym^n(E_1\times E_2),t) u^n\\
&=& [[\;[1]\;]]- [[\;[\alpha_1]\;]] - [[\;[\alpha_2]\;]] - [[\;[\beta_1]\;]] - [[\;[\beta_2]\;]] \\
&&+ [[\;[\alpha_1\beta_2]\;]]+[[\;[\alpha_2\beta_1]\;]] +[[\;[\alpha_1\alpha_2]\;]] + [[\;[\beta_1\beta_2]\;]] \\
&&- [[\;[\alpha_1q]\;]] -[[\; [\alpha_2q]\;]] - [[\;[\beta_1q]\;]] - [[\;[\beta_2q]\;]] + [[\;[q^2]\;]].
\end{eqnarray*}
Note that these equalities are being written in $W(W(A))$.  These methods similarly also work for $n$-fold products of smooth projective curves; however, the formulas quickly become unwieldy.

\section{Lambda-ring valued motivic measures} \label{lambdaringmeasures}
Let $k$ be a field and $\Var_k$ be the category of varieties over $k$ (reduced schemes of finite type over $\Spec k$).  The \emph{Grothendieck ring of varieties} $\GK$ is the abelian group generated by symbols $[X]$ of isomorphism classes of $X \in \Var_k$ subject to the scissor relation $[X] = [Y] + [X \setminus Y]$ for $Y$ any closed subvariety of $X$.  $\GK$ is a commutative ring under the product $[X]\cdot[Y] = [X \times_k Y]$.  This is the so-called universal value group of Euler-Poincar\'{e} characteristics on $\Var_k$; ring homomorphisms on $\GK$ are called \emph{motivic measures}.  That is, for $A$ a commutative ring with identity, we consider $\mu:\GK \rightarrow A$.  Kapranov in \cite{Kapranov} constructs a \emph{motivic zeta function} $\zeta_\mu(X,t)$ for $X \in \Var_k$ using these algebraic invariants $\mu(X)$.  For $X$ quasi-projective, it is the generating series for (the measure of) symmetric powers $\Sym^n X$:
$$
\zeta_\mu(X,t) = \sum_{n=0}^\infty \mu(\Sym^n X) t^n \in A[[t]].
$$
As $\GK$ is additively generated by quasiprojective $X$, this defines a group homomorphism on $\GK$ taking values in the group of invertible power series $\Lambda(A)$.  If $\zeta_\mu:\GK \rightarrow W(A)$ is a ring homomorphism , we say that $\mu$ exponentiates (see \cite{Ramachandran15}).

For example, in the finite field case\footnote{In this case, we must pass through the $\widetilde{K_0}(\Var_{\F_q})$ Grothendieck ring modulo \emph{radicial morphisms}, but all the measures on $\Var_{\F_q}$ we consider factor through $\widetilde{K_0}(\Var_{\F_q})$. See Mustata \cite[p. 78]{Mustata}.} $X \in \Var_{\F_q}$, the motivic zeta function $\zeta_{\mu_\#}$ associated to  counting measure $\mu_\#([X]) = \#X(\F_q)$  recovers the usual zeta function $Z(X,t)$ (see Mustata \cite[Prop 7.31]{Mustata}).  We have seen that $Z(X\times_k Y,t) = Z(X,t)*_W Z(Y,t)$; hence $\mu_\#$ can be exponentiated.

We are particularly interested in $A$-valued motivic measures where $A$ is a $\lambda$-ring.  Note that $\GK$ has a pre-$\lambda$-ring structure provided by the Kapranov motivic zeta function, defined such that $\sigma^n([X]) = [\Sym^n X]$ for $X$ quasiprojective.  This, however, is not a $\lambda$-ring  structure (see Larsen and Lunts \cite{LarsenLunts}).

\begin{definition}
Let $A$ be a $\lambda$-ring.  A motivic measure $\mu:\GK \rightarrow A$ is called a $\lambda$-measure if any of the equivalent conditions hold:
\begin{itemize}
\item $\mu$ is a pre-$\lambda$-ring map, $\mu(\Sym^n[X]) = \sigma^n \mu(X)$.
\item The associated motivic zeta function factors as
$$\label{diagram}
\xymatrixcolsep{3pc}\xymatrix{
K_0(\Var_k) \ar[rd]_{\zeta_\mu} \ar[r]^\mu
&R \ar[d]^{\sigma_t}\\
&W(R)}	
$$
\end{itemize}
\end{definition}

\begin{proposition}[Ramachandran, Tabuada \cite{RamaTabuada}]\label{NRTaba}
If $\mu$ is a $\lambda$-measure, then $\mu$ exponentiates.
\end{proposition}
\begin{proof}
The composition above is a ring homomorphism.
\end{proof}

\begin{example}
The generating series in (\ref{macd}) for the Poincar\'{e} polynomial is the Kapranov motivic zeta function $\zeta_{\mu_P}$ associated to the Poincar\'{e} polynomial measure $\mu_P(X) = P(X,z)$.  Since the ghost coordinates $\gh_n(\zeta_{\mu_P}(X,t)) = P(X,z^n)$ are multiplicative, $\zeta_P$ takes values in the Witt ring.  MacDonald's formula (\ref{macd}) can be written in the Witt ring $W(\Z[z])$ as
$$
\zeta_{\mu_P}(X,t) =  \sum_{i=0}^{2m} (-1)^i b_i(X) [z^i].
$$
The zeta function for the Poincar\'{e} polynomial appears as a Poincar\'{e} polynomial written in Teichm\"{u}ller elements.  Similarly for $\chi$, we have
$$
\zeta_\chi(X,t) = \sum_{i=0}^{2m} (-1)^i b_i(X) [1] = \chi(X)[1],
$$
an Euler characteristic written in Teichm\"{u}ller elements.  Moreover, both these motivic zeta functions factor through the $\lambda$-ring structure maps and thus both $\chi$ and $\mu_P$ are $\lambda$-measures:
\begin{proposition}
For the Euler characteristic, $\zeta_\chi = \sigma_t \circ \chi$,  where $\sigma_t$ is the opposite $\lambda$-ring structure map on $\Z$ and for the Poincar\'{e} polynomial $\zeta_{\mu_P} = \sigma_t' \circ \mu_P$, where $\sigma_t'$ is the opposite $\lambda$-ring structure map on $\Z[z]$.
\end{proposition}
\begin{proof}
 Recall that for $a \in \Z$, $\sigma_t(a) = a[1]$ and for $f(z) \in \Z[z]$, $\sigma_t'(f(z)) = f([z])$.
\end{proof}
\end{example}

\begin{definition}
Given a measure that exponentiates $\mu:\GK \rightarrow A$, the associated motivic zeta function $\zeta_\mu:\GK \rightarrow W(A)$ is a ring homomorphism.  The induced motivic measure is called the \emph{zeta function measure} $\mu_Z(X) = \zeta_\mu(X,t)$; it is a $W(A)$-valued motivic measure.
\end{definition}

\begin{theorem}\label{zetaofzeta}
If $\mu:\GK \rightarrow A$ is a $\lambda$-measure, then its induced zeta function measure $\mu_Z:\GK \rightarrow W(A)$ is also a $\lambda$-measure.  This process iterates ad infinitum.
\end{theorem}
\begin{proof}
By Proposition \ref{sigma}, the map $\sigma_t$ is a $\lambda$-ring map.  The composition above is a $\lambda$-ring map.
\end{proof}

\begin{remark}
The zeta function $Z(X,t)$ is the Kapranov motivic zeta function associated to the counting measure $\mu_\#$; however, counting measure is not a $\lambda$-measure.  Recall the unique $\lambda$-ring structure on $\Z$: for $a \in \Z$,
$$
\lambda_t(a) =(1+t)^a, \quad \lambda^n(a) = \left(\! \begin{array}{c} a \\ n \end{array} \! \right),
\quad \sigma^n(a) = \left(\! \begin{array}{c} a+n-1 \\ n \end{array} \! \right).
$$
However, it is easy to verify that $\mu_\#(\Sym^n X) \neq \sigma^n (\mu_\#(X))$.  For instance, $X = \A^1$, $\Sym^n \A^1 = \A^n$ and $\mu_\#(\A^1) = q$ whereas $\mu_\#(\A^n) = q^n \neq \sigma^n(q)$.  Thus, the exponentiability of the induced zeta function measure $\mu_Z$ does not follow from the above Theorem.  However,

\begin{corollary}[Corollary to Theorem \ref{mainthm}]
Let $Z(X,t)$ be the zeta function for $X \in \Var_{\F_q}$ and $\mu_Z$ its induced zeta function measure.  Then
$$
\zeta_{\mu_Z}(X,u) = \bm{\sigma_u}(Z(X,t)) \in W(W(\Z))
$$
and so $\mu_Z$ is a $\lambda$-measure.
\end{corollary}

\begin{proof}
The formula in Theorem \ref{mainthm} shows that for $X \in \Var_{\F_q}$
$$
\zeta_{\mu_Z}(X,u) =  \sum_{i,j} (-1)^{i+1}[[\; [\alpha_{ij}] \;]].
$$
Recall that the map $\bm{\sigma_u}$ on Teichm\"{u}ller elements is given by the double Teichm\"{u}ller $\bm{\sigma_u}([\alpha]) = [[\;[\alpha]\;]]$.  As $\bm{\sigma_u}$ is a ring homomorphism, we have $\zeta_{\mu_Z} = \bm{\sigma_u}(Z(X,t))$.
\end{proof}
\end{remark}

\begin{remark}
The formula for $Z(\Sym^n X,t)$ in Theorem \ref{mainthm} is analogous to the MacDonald formulae above, as it is simply the zeta function written in (double) Teichm\"{u}ller elements.  And similarly, the formula shows that $\mu_Z$ is a $\lambda$-measure.
\end{remark}

\bibliographystyle{amsplain}
\bibliography{mybib}

\end{document}